\documentclass[12pt]{amsart}
\usepackage{stmaryrd}
\usepackage{mathrsfs}
\usepackage{amsmath,amssymb}
\usepackage{amsfonts}
\usepackage{amsthm}
\usepackage{latexsym}
\usepackage{graphicx,color}
\def\Rc{\mbox{Rc}}

\def\p{\partial}

\def\vv<#1>{\langle#1\rangle}

\def\1{\mathbf{1}}

\def\XXint#1#2{\setbox0=\hbox{$#1{#2}{\int}$}{#2}\kern-.5\wd0 }

\def\XXint#1#2#3{{\setbox0=\hbox{$#1{#2#3}{\int}$}
     \vcenter{\hbox{$#2#3$}}\kern-.5\wd0}}

\def\vv<#1>{{\left\langle#1\right\rangle}}

\def\Vol{{\rm Vol}}
\def\sph{\mathbb{S}}

\def\e{\epsilon}

\def\Rc{{\rm Rc}}
\def\Hess{{\rm Hess}}
\def\W{{\mathcal W}}
\def\Scal{{\rm Scal}}
\newtheorem{thm}{Theorem}[section]

\newtheorem{prop}{Proposition}[section]

\theoremstyle{definition}

\theoremstyle{remark}

\numberwithin{equation}{section}

\begin{document}
\title{A note on the equality case of Bray's conjecture}

\author{Xu-Qian Fan}
\address{Department of Mathematics, Jinan University, Guangzhou, 510632, China}
\email{txqfan@jnu.edu.cn}
\author{Chengjie Yu$^\dagger$}
\address{Department of Mathematics, Shantou University, Shantou, Guangdong, 515063, China}
\email{cjyu@stu.edu.cn}
%\thanks{$^1$**** }
\thanks{$^\dagger$ Research partially supported by GDNSF with contract no. 2025A1515011144 and 2026A1515012267.}
\renewcommand{\subjclassname}{%
  \textup{2020} Mathematics Subject Classification}
\date{}
\begin{abstract}
In this short note, we characterize the equality case of Bray's conjecture recently proved by Jiang-Li-Wang \cite{JLW}.
\end{abstract}
\maketitle
\markboth{Fan \& Yu}{Rigidity of Bray's conjecture}
\section{Introduction}
In a recent preprint \cite{JLW}, Jiang-Li-Wang \cite{JLW} proved the following result which is equivalent to Bray's conjecture \cite{Br}.

\begin{thm}[Jiang-Li-Wang \cite{JLW}]\label{thm-JLW}
Let $n\geq 3$ be an integer. Then, there is a positive constant $\e_n$ such that for any $\e\in (0,\e_n]$, 
if the closed Riemannian manifold $(M^n,g)$ satisfies the curvature assumptions:
\[
\operatorname{Rc}_g \ge (n-1)g \mbox{ and } \operatorname{Scal}_g \ge n(n-1)(1+\epsilon),
\]
 then
\[
\operatorname{Vol}_g(M) \leq  |\mathbb{S}^n| (1+\epsilon)^{-n/2}.
\]
Here $\Rc_g$ and $\operatorname{Scal}_g$ are the Ricci tensor and scalar curvature of $g$ respectively, and $\Vol_g(M)$ is the volume of $(M,g)$. 
\end{thm}

There are two main ingredients in the proof of Theorem \ref{thm-JLW} by Jiang-Li-Wang \cite{JLW}:
\begin{enumerate}
\item A comparison of Perelman's $\nu$-entropies between closed Riemannian manifolds satisfying the curvature assumptions in Theorem \ref{thm-JLW} and the standard unit sphere: By a P\'olya-Szeg\"o theorem for Schwartz symmetrization on spheres (see \cite{MS}) and Beckner's sharp logarithmic Sobolev inequality on spheres (see \cite{Be}), Jiang-Li-Wang obtained the following comparison of entropies:
\begin{equation}\label{eq-comp-entropy}
\nu(g)\geq \nu(g_{\sph^n})+\ln\frac{\operatorname{Vol}_g(M)}{|\sph^n|}+\frac{n}{2}\ln(1+\e);
\end{equation}
\item Ma-Wang's stability theorem for  normalized Ricci flow (see \cite{MW}): For any $p>\frac{n}{2}$, there is a positive constant $\delta_{\rm MW}(n,p)$ such that for any closed Riemannian manifold $(M^n,g)$ with 
\begin{equation}
\operatorname{Vol}_g(M)=|\sph^n|
\end{equation}
and
\begin{equation}\label{eq-stability}
\int_M((n-1)-\lambda_g(x))_+^pdV_g<\delta_{\rm MW}(n,p),
\end{equation} 
where $\lambda_g(x)$ is the minimal eigenvalue of $\Rc_g(x)$ for any $x\in M$, the normalized Ricci flow with initial data $g$ is immortal and converges exponentially to a metric of constant sectional curvature one.
\end{enumerate}
The constant $\e_n$ in Theorem \ref{thm-JLW}  is given by 
\begin{equation}\label{eq-e-n}
\e_n=\min\left\{\frac{n}{n-1}e^{-\frac1n}-1, \left(\frac{\delta_{\rm MW}(n,n)}{|\sph^n|(n-1)^n}\right)^\frac1n\right\}
\end{equation}
in \cite{JLW}. The requirement that 
$$\e_n\leq\frac{n}{n-1}e^{-\frac1n}-1$$ 
is to guarantee that \eqref{eq-comp-entropy} is true, and the requirement that $$\e_n\leq\left(\frac{\delta_{\rm MW}(n,n)}{|\sph^n|(n-1)^n}\right)^\frac1n$$
is to guarantee that \eqref{eq-stability} is true so that Ma-Wang's stability theorem can be applied. 

In this short note, we characterize the equality case of Theorem \ref{thm-JLW}. More precisely, we prove the following result.
\begin{thm}\label{thm-main}
Let the notations and assumptions be the same as in Theorem \ref{thm-JLW}. Moreover, suppose that  
\[
\operatorname{Vol}_g(M) =|\mathbb{S}^n| (1+\epsilon)^{-n/2}.
\]
Then, $(M,g)$ is isometric to the round sphere of radius $(1+\e)^{-1/2}$.
\end{thm}
The idea of the proof of the result is really simple. When the equality of the volume holds, combining the monotonicity of $\nu$-entropy along Ricci flow, the entropy comparison \eqref{eq-comp-entropy} and Ma-Wang's stability theorem, one knows that $\nu(g(t))$ is constant along the Ricci flow starting from $h:=(1+\e)g$. Then, by Perelman's formula for derivative of $\W$-entropy, we know that $g$ is a gradient shrinking soliton. From this, it is not hard to get the conclusion by Ma-Wang's stability theorem again. 

Bray's conjecture was first proposed and studied in Bray's thesis \cite{Br}. Further progresses can be found in \cite{GV,Bre,Yu,Zh}. Very recently, Kwong \cite{Kw} obtained a related result: Let $(M^n,g)$ be a closed Riemannian manifold with 
$$\Rc_g\geq(n-1)g\mbox{ and }\Scal_g\geq n(n-1)(1+\e)$$
for some $\e\geq0$. Then,
$$\Vol_g(M)\leq |\sph^n|(1+n\e)^{-\frac12}$$
with equality if and only if $\e=0$ and $(M,g)$ is isometric to the standard unit sphere. One should note that although the upper bound of Kwong's result is greater than that of Theorem \ref{thm-JLW}, Kwong's result does not require that $\e$ is small enough. The method used by Kwong \cite{Kw} is different with that of Jiang-Li-Wang \cite{JLW}.
\section{Preliminaries and proof}
In this section, we recall some preliminaries on Perelman's $\W$-entropy and $\nu$-entropy, and prove Theorem \ref{thm-main}.

Let $M^n$ be a closed manifold, $g$ be a Riemannian metric on $M$, $f\in C^\infty(M)$ and $\tau>0$. Then, the $\mathcal W$-entropy of Perelman is defined as 
$$\W(g,f,\tau)=\int_M\left(\tau(\|\nabla f\|^2+\Scal_g)+f-n\right)(4\pi\tau)^{-\frac{n}{2}}e^{-f}dV_g.$$
The following formula for the derivative of $\W$ along Ricci flow is an important   discovery of Perelman (see Perelman \cite{Pe} or Cao-Zhu \cite[P. 206, Proposition 1.5.8]{CZ} for more details). 
\begin{thm}\label{thm-derivative-W}
Let $M^n$ be a closed manifold and $g(t)$ be a Ricci flow on $M$. Let $\tau(t)$ be a positive function and $f(t)$ be a family of smooth functions on $M$ such that $\frac{d\tau}{dt}=-1$ and $u(x,t):=(4\pi\tau(t))^{-\frac{n}{2}}e^{-f(x,t)}$ satisfies the conjugate heat equation:
$$\frac{\p u}{\p t}=-\Delta_{g(t)} u+\Scal_{g(t)}u.$$
Then,
\begin{equation}\label{eq-derivative-W}
\begin{split}
&\frac{d}{dt}\mathcal W(g(t),f(t),\tau(t))\\
=&2\tau(t)\int_M\left\|\Rc_{g(t)}+\Hess_{g(t)}(f(t))-\frac{1}{2\tau(t)}g(t)\right\|_{g(t)}^2(4\pi\tau(t))^{-\frac{n}{2}}e^{-f(t)}dV_{g(t)}
\end{split}
\end{equation}
where $\Delta_g$ and $\Hess_g$ are the Laplacian operator and the Hessian operator w.r.t. $g$ and $\|\cdot\|_g$ means taking norm w.r.t. $g$.
\end{thm}
Let 
$$\mu(g,\tau)=\inf\left\{\mathcal W(g,f,\tau)\ \bigg|\ \int_M(4\pi\tau)^{-\frac{n}{2}}e^{-f}dV_g=1\right\}$$
and 
$$\nu(g)=\inf_{\tau>0}\mu(g,\tau).$$
It is well known that if $g(t)$ is a Ricci flow, then $\nu(g(t))$ is increasing. Note that 
\begin{equation}\label{eq-scale-inv}
\nu(cg)=\nu(g),\ \forall\ c>0.
\end{equation}
So, $\nu(g(t))$ is also increasing when $g(t)$ is a normalized Ricci flow. More precisely, one has the following conclusion.
\begin{prop}\label{prop-derivative-nu}
Let $M^n$ be a closed manifold and $g(t)$ with $t\in [a,b]$ be a Ricci flow on $M$. Then, for any $t_0\in (a,b]$, 
\begin{equation}\label{eq-derivative-nu}
D_-\nu(t_0)\geq 2\tau_0\int_M\left\|\Rc_g(t_0)+\Hess_{g(t_0)}(f_0)-\frac{1}{2\tau_0}g(t_0)\right\|_{g(t)}^2(4\pi \tau_0)^{-\frac{n}{2}}e^{-f_0}dV_{g(t_0)}
\end{equation}
where $(f_0,\tau_0)$ is the minimizer so that 
$$\nu(g(t_0))=\W(g(t_0),f_0,\tau_0).$$
Here $\nu(t):=\nu(g(t))$ and 
$$D_-\nu(t_0)=\liminf_{t\to t_0^-}\frac{\nu(t)-\nu(t_0)}{t-t_0}.$$
In particular, if there is a $t_0\in (a,b]$ such that $\nu'(t_0)=0$, then $g(t)$ is a gradient shrinking soliton. 
\end{prop}
\begin{proof}
Let $u_0=(4\pi \tau_0)^{-\frac{n}{2}}e^{-f_0}$, $u(x,t)$ be the solution to the conjugate heat equation:
$$\frac{\p u}{\p t}=-\Delta_{g(t)} u+\Scal_{g(t)} u$$
with $u(t_0)=u_0$ and $$\tau(t)=\tau_0+t_0-t.$$
Let $f(x,t)$ be such that $u(x,t)=(4\pi \tau(t))^{-\frac{n}{2}}e^{-f(x,t)}$. Then
$$\W(g(t),f(t),\tau(t))\geq\nu(t)$$
for $t<t_0$, by the definition of $\nu$, and 
$$\W(g(t_0),f(t_0),\tau(t_0))=\nu(t_0)$$
by that $f(t_0)=f_0$ and $\tau(t_0)=\tau_0$. So, 
$$\frac{\W(g(t),f(t),\tau(t))-\W(g(t_0),f(t_0),\tau(t_0))}{t-t_0}\leq \frac{\nu(t)-\nu(t_0)}{t-t_0}$$
for $t<t_0$. Applying $\displaystyle\liminf_{t\to t_0^-}$ to the both sides of the inequality above and using Theorem \ref{thm-derivative-W}, we get \eqref{eq-derivative-nu}. This completes the proof of the proposition.
\end{proof}
Finally, we prove Theorem \ref{thm-main}.
\begin{proof}[Proof of Theorem \ref{thm-main}]
Substituting $\operatorname{Vol}_g(M)=|\sph^n|(1+\e)^{-\frac{n}{2}}$ into  \eqref{eq-comp-entropy}, one has
\begin{equation}\label{eq-entropy-g}
\nu(g)\geq \nu(g_{\sph^n}).
\end{equation}
Let \(h := (1+\epsilon)g\). Then
\[
\operatorname{Vol}_h(M) = (1+\epsilon)^{n/2} \operatorname{Vol}_g(M) = |\mathbb{S}^n|.
\]
and 
\begin{equation}\label{eq-entropy-h}
\nu(h)\geq \nu(g_{\sph^n})
\end{equation}
by \eqref{eq-scale-inv} and \eqref{eq-entropy-g}. Moreover, since the Ricci tensor is invariant under constant rescaling,
\[
\operatorname{Rc}_h = \operatorname{Rc}_g \ge (n-1)g = \frac{n-1}{1+\epsilon}\,h.
\]
So, $\lambda_h\geq \frac{n-1}{1+\epsilon}$ and 
$$(n-1)-\lambda_h\leq \frac{(n-1)\e}{1+\e}<(n-1)\e$$
which implies that 
$$\int_M((n-1)-\lambda_h(x))_+^ndV_h<(n-1)^n|\sph_n|\e^n\leq \delta_{\rm MW}(n,n)$$
by that $\e\leq\e_n$ with $\e_n$ in \eqref{eq-e-n}. Let $h(t)$ be the normalized Ricci flow with $h(0)=h$. Then, Ma-Wang's stability theorem asserts that $h(t)$ is immortal and 
$$h(t)\to h_\infty\ (t\to\infty)$$
exponentially with $h_\infty$ a metric of constant sectional curvature one. Because $$\operatorname{Vol}_{h_\infty}(M)=\operatorname{Vol}_h(M)=|\sph^n|,$$
we know that $(M,h_\infty)$ is isometric to $(\sph^n,g_{\sph^n})$. Thus 
$$\nu(h_\infty)=\nu(g_{\sph^n}).$$
By that $\nu(h(t))$ is increasing on $t$, we have 
$$\nu(h)\leq \nu(h_\infty)=\nu(g_{\sph^n}).$$ 
Combining this with \eqref{eq-entropy-h}, one has 
$$\nu(h)=\nu(g_{\sph^n})=\nu(h_\infty).$$
By the monotonicity of $\nu(h(t))$ again,
$$\nu(h(t))=\nu(g_{\sph^n})$$
for all $t\geq 0$. Let $g(t)$ be the Ricci flow on $M$ with $g(0)=h$. By that $\nu$ is invariant under constant rescaling again, $\nu(g(t))=\nu(g_{\sph^n})$ is constant. By Proposition \ref{prop-derivative-nu}, $g(t)$ is a gradient shrinking soliton. This implies that for any $t_2>t_1\geq 0$, there is a positive constant $\alpha$ and a diffeomorphism $\varphi$ such that 
$$h(t_2)=\alpha \varphi^*h(t_1).$$
Note that $\operatorname{Vol}_{h(t)}(M)$ is constant along normalized Ricci flow. So $\alpha=1$. This means that $(M,h(t_1))$ and $(M,h(t_2))$ are isometric for any $t_2>t_1\geq0$. Then, by Ma-Wang's stability theorem again and that $(M,h_\infty)$ is isometric to $(\sph^n,g_{\sph^n})$, $(M,h(t))$ is isometric to $(\sph^n,g_{\sph^n})$ for any $t\geq 0$. This completes the proof of the theorem. 
\end{proof}
\subsection*{Acknowledgments}
The authors used DeepSeek as an AI-assisted tool in the preparation 
of this note and take full responsibility for the contents.

\end{document}